\documentclass[a4paper,11pt]{amsart}
\usepackage{mathrsfs,nccmath,amssymb,amsthm,mathabx,mathtools}
\usepackage{graphicx,psfrag}
\usepackage[all]{xy}
\usepackage[all]{xy}
\xyoption{line}
\usepackage{color}
\usepackage{ulem}

\renewcommand{\le}{\varleq}
\renewcommand{\ge}{\vargeq}

\newcommand\mL{L\kern-0.08cm\char39}

\newcommand{\myand}{\text{ and }}

\newcommand{\seb}{\{\,}
\newcommand{\sen}{\,\}}

\newcommand{\getsby}[1]{\xleftarrow{#1}}

\newcommand{\tesgh}{edge-surjective graph homomorphism}

\newcommand{\pdirectional}{\raise0.05em\hbox{$+$}directional}
\newcommand{\pdirectionality}{\raise0.05em\hbox{$+$}directionality}
\newcommand{\pdirectionalitys}{\raise0.05em\hbox{$+$}directionality }
\newcommand{\pdirectionals}{\raise0.05em\hbox{$+$}directional }
\newcommand{\bidirectional}{bidirectional}
\newcommand{\bidirectionals}{bidirectional }
\newcommand{\bidirectionality}{bidirectionality}
\newcommand{\bidirectionalitys}{bidirectionality }

\newcommand{\Z}{\mathbb{Z}}
\newcommand{\Nonne}{\mathbb{N}}
\newcommand{\Posint}{\mathbb{N} \setminus \{ 0 \}}

\newcommand{\bi}{\in \Z}

\newcommand{\beposint}{\in \Posint}
\newcommand{\bpi}{\ge 1} 

\newcommand{\benonne}{\in \Nonne}
\newcommand{\bni}{\ge 0} 

\newcommand{\diam}{{\rm diam}}

\newcommand{\Fcal}{\mathcal{F}}
\newcommand{\Gcal}{\mathcal{G}}

\newcommand{\Ucal}{\mathcal{U}}

\newcommand{\kuu}{\emptyset}
\newcommand{\nekuu}{\neq \kuu}
\newcommand{\fai}{\varphi}

\newcommand{\hx}{\hat{x}}
\newcommand{\hy}{\hat{y}}

\newcommand{\hY}{\hat{Y}}

\newcommand{\centb}{\begin{center}}
\newcommand{\centn}{\end{center}}

\newcommand{\enumb}{\begin{enumerate}}
\newcommand{\enumn}{\end{enumerate}}

\newcommand{\itemb}{\begin{itemize}}
\newcommand{\itemn}{\end{itemize}}

\newtheorem{thm}{Theorem}[section]
\newtheorem{lem}[thm]{Lemma}

\newtheorem{mainthm}{Theorem}

\theoremstyle{definition}
\newtheorem{defn}[thm]{Definition}

\theoremstyle{remark}
\newtheorem{nota}[thm]{Notation}
\newtheorem{rem}[thm]{Remark}

\numberwithin{equation}{section}

\newcommand{\covrep}[2]{#1_0 \getsby{#2_1} #1_1 \getsby{#2_2} #1_2 \getsby{#2_3} \dotsb}

\begin{document}

\title[Combinatorial embedding of chain transitive]
{Combinatorial embedding of chain transitive zero-dimensional systems into chaos}

\author{TAKASHI SHIMOMURA}

\address{Nagoya University of Economics, Uchikubo 61-1, Inuyama 484-8504, Japan}
\curraddr{}
\email{tkshimo@nagoya-ku.ac.jp}
\thanks{}



\date{\today}

\dedicatory{}

\commby{}

\begin{abstract}
We show that a zero-dimensional chain transitive dynamical system can be embedded into
a densely uniformly chaotic system, with dense uniformly chaotic set $K$.
We concretely construct a Mycielski set $K$ that is also invariant.
Furthermore, every point in $K$ is positively and negatively transitive.
The uniform proximality and recurrence of $K$ are also bidirectional.
\end{abstract}

\maketitle
\section{Introduction}
A pair $(X,f)$ of a compact metric space $X$ and a continuous surjection $f : X \to X$
 is called a {\it topological dynamical system}.
A topological dynamical system $(X,f)$ is called a zero-dimensional system,
 if $X$ is totally disconnected.
In \cite{Shimomura4}, we presented a way to express every zero-dimensional system
 combinatorially by a sequence of graph coverings (see \S \ref{sec:covering}).
Let $(Y,g)$ be an arbitrary chain transitive zero-dimensional system.
In this paper, we construct a chaotic zero-dimensional system
 $(X,f)$ that contains $(Y,g)$ and has a dense scrambled set.
A topological dynamical system $(X,f)$ is said to be {\it chain transitive} if
 for any $\epsilon > 0$ and any pair $(x,y) \in X \times X$, there exists
 a finite sequence $(x_0 = x, x_1,x_2,\dotsc,x_l)$ such that 
 $d(f(x_i),x_{i+1}) < \epsilon$ for all $0 \le i < l$.
A pair $(x,y) \in X^2\setminus \Delta_X$ is said to be proximal, if 
\[ \liminf_{n \to +\infty}d(f^n(x),f^n(y)) = 0.\]
A pair $(x,y)$ is said to be a Li--Yorke pair if $(x,y)$ is proximal and
 $(x,y)$ satisfies
\[ \limsup_{n \to +\infty}d(f^n(x),f^n(y)) > 0.\]
A subset $K \subset X$ is said to be {\it scrambled}\/ if any pair
 $(x,y) \in K^2\setminus \Delta_K$ is a Li--Yorke pair.
If there exists an uncountable scrambled set, then the system is called
 Li--Yorke chaotic system.
Note that if a proximal pair $(x,y)$ with $x \ne y$ is recurrent
 in the sytem $(X\times X,f\times f)$,
 then it is a Li--Yorke pair.
Akin {\it et al.}\/ \cite{AGHSY} presented a stronger chaos notion 
 called uniform chaos, and also
 presented a criterion for chaos (Theorem 3.1 of \cite{AGHSY}).
Actually, they posed the notion of {\it uniformly chaotic set}
 (see Definition \ref{defn:uniformly-chaotic-set}).
The system that has such set is said to be uniformly chaotic, and
 the system that has a dense uniformly chaotic set is said to be
 {\it densely uniformly chaotic}.
On the other hand, Yuan and L\"{u} in \cite{YL}, and Tan in \cite{Tan},
 without the assumption of compactness,
 investigated into the invariance of scrambled sets.
We show the following:
\begin{mainthm}\label{thm:main}
Let $(Y,g)$ be a chain transitive zero-dimensional system.
Then, there exists densely uniformly chaotic
 zero-dimensional system  $(X,f)$ with a dense uniformly chaotic invariant set
 $K \subset X\setminus Y$ such that
\enumb
\item\label{main:homeomorphism} the restriction $f_{X \setminus Y}$ is a homeomorphism,
\item\label{main:K} $K = \bigcup_{N \ge 1} C_N$,
 where $C_1 \subset C_2 \subset \dotsb$ is an 
 increasing sequence of Cantor sets,
\item\label{main:proximal}
   each $C_N$ is both positively and also negatively uniformly proximal,
\item\label{main:recurrent}
 each $C_N$ is both positively and also negatively uniformly recurrent,
\item\label{main:invariant} $f(K) = K$,
\item\label{main:dense} $K$ is dense in $X$,
\item\label{main:transitive}
 each $x \in K$ is positively and also negatively transitive.
\enumn
\end{mainthm}
By \ref{main:K}, \ref{main:proximal}, \ref{main:recurrent}, \ref{main:invariant},
 and \ref{main:dense}, $K$ is the dense uniformly chaotic invariant set.
%
%
%
%
%
\section{Preliminaries.}
Let  $\Z$ be the set of all integers,
 and $\Nonne$ be the set of all non-negative integers.
For integers $a < b$, the intervals are denoted by $[a,b] := \seb a, a+1, \dotsc,b \sen$.
%
%
\subsection{Uniformly chaotic set.}
In this subsection, we introduce the notion presented by Akin {\it et al.}\/
 in \cite{AGHSY}.
Let $(X,f)$ be a topological dynamical system.
A subset $A \subset X$ is {\it uniformly recurrent}\/
 if for every $\epsilon > 0$ there exists an arbitrarily large $n > 0$ such that
 $d(f^n(x),x) < \epsilon$ for all $x \in A$.
A subset $A \subset X$ is {\it uniformly proximal}\/
 if $\liminf_{n \to +\infty}\diam(A) = 0$.
\begin{defn}[Akin {\it et al.} \cite{AGHSY}]\label{defn:uniformly-chaotic-set}
Let $(X,f)$ be a topological dynamical system.
A subset $K \subset X$ is called a {\it uniformly chaotic set}\/ if
 there exist Cantor sets $C_1 \subset  C_2 \subset \dotsb$ such that:
\enumb
\item $K = \bigcup_{i = 1}^{\infty}C_i$,
\item for each $N \bni$, $C_N$ is uniformly recurrent, and
\item for each $N \bni$, $C_N$ is uniformly proximal.
\enumn
Here, $(X,f)$ is called (densely) uniformly chaotic if
 $(X,f)$ has a (dense) uniformly chaotic subset.
\end{defn}
\begin{rem}
Our definition of the uniformly chaotic set seems to be different a little
 formally.
Nevertheless, there exists no difference, because of the next two facts:
\itemb
\item if $K$ is a uniformly chaotic set, then
 every finite subset of $K$ is uniformly recurrent,
\item if $K$ is a uniformly chaotic set, then
 every finite subset of $K$ is uniformly proximal.
\itemn
For the original definition, see \S 2.1 of \cite{AGHSY}.
\end{rem}
Every uniformly chaotic set is a scrambled set in the sense of Li--Yorke.
For further discussions, see \cite{AGHSY}.
\subsection{Graph covering.} \label{sec:covering}
%
%
We have given the notion of graph coverings for
 all zero-dimensional continuous surjections (see \cite[\S 3]{Shimomura4}).
In this section,
 we repeat the construction of general graph coverings
 for general zero-dimensional systems.
A pair $G = (V,E)$
 consisting of a finite set $V$ and a relation $E \subseteq V \times V$ on $V$
 can be considered as a directed graph with vertices $V$
 and an edge from $u$ to $v$ when $(u,v) \in E$.
For a finite directed graph $G = (V,E)$, we write as $V = V(G)$ and $E = E(G)$.
%
%
\begin{nota}
In this paper, we assume that a finite directed graph $G = (V,E)$
 is a surjective relation,
 i.e., for every vertex $v \in V$ there exist edges $(u_1,v),(v,u_2) \in E$.
\end{nota}
For finite directed graphs $G_i = (V_i,E_i)$ with $i = 1,2$,
 a map $\fai : V_1 \to V_2$ is said to be a {\it graph homomorphism}\/
 if for every edge $(u,v) \in E_1$, it follows that $(\fai(u),\fai(v)) \in E_2$.
In this case, we write as $\fai : G_1 \to G_2$.
For a graph homomorphism $\fai : G_1 \to G_2$,
 we say that $\fai$ is {\it edge-surjective}\/
 if $\fai(E_1) = E_2$.
%
%
%
%
%
Suppose that a graph homomorphism $\fai : G_1 \to G_2$ satisfies the following condition:
\[(u,v),(u,v') \in E_1 \text{ implies that } \fai(v) = \fai(v').\]
In this case, $\fai$ is said to be {\it \pdirectional}.
Suppose that a graph homomorphism $\fai$ satisfies both of the following conditions:
\[(u,v),(u,v') \in E_1 \text{ implies that } \fai(v) = \fai(v') \myand \]
\[(u,v),(u',v) \in E_1 \text{ implies that } \fai(u) = \fai(u').\]
Then, $\fai$ is said to be {\it \bidirectional}.
%
%
\begin{defn}\label{defn:cover}
For a finite directed graphs $G_1$ and $G_2$,
 a graph homomorphism $\fai : G_1 \to G_2$
 is called a {\it cover}\/ if it is a \pdirectionals \tesgh.
\end{defn}
For a sequence $G_1 \getsby{\fai_2} G_2 \getsby{\fai_3} \dotsb$
 of graph homomorphisms and $m > n$,
 we write $\fai_{m,n} := \fai_{n+1} \circ \fai_{n+2} \circ \dotsb \circ \fai_m$.
Then, $\fai_{m,n}$ is a graph homomorphism.
If all ${\fai_i}$ $(i \beposint)$ are edge surjective,
 then every $\fai_{m,n}$ is edge surjective.
If all ${\fai_i}$ $(i \beposint)$ are covers, every $\fai_{m,n}$ is a cover.
%
%
Let $G_0 := \left(\seb v_0 \sen, \seb (v_0,v_0) \sen \right)$ be a singleton graph.
For a sequence of graph covers $G_1 \getsby{\fai_2} G_2 \getsby{\fai_3} \dotsb$, we
 attach the singleton graph $G_0$ at the head.
We call a sequence of graph covers
 $\covrep{G}{\fai}$
 as a {\it graph covering}\/ or just a {\it covering}.
From this paper, considering the numbering of Bratteli diagrams,
 we use this numbering of graph covering.
In the original paper, we have used the numbering as $G_n \getsby{\fai_n} G_{n+1}$.
%
%
Let us write the directed graphs as $G_i = (V_i,E_i)$ for $i \benonne$.
We define the {\it inverse limit}\/ of $\Gcal$ as follows:
\[V_{\Gcal} := \seb (v_0,v_1,v_2,\dotsc)
 \in \prod_{i = 0}^{\infty}V_i~|~v_i = \fai_{i+1}(v_{i+1})
 \text{ for all } i \benonne \sen \text{ and}\]
\[E_{\Gcal} := 
\seb (x,y) \in V_{\Gcal} \times V_{\Gcal}~|~
(u_i,v_i) \in E_i \text{ for all } i \benonne\sen,\]
where $x = (u_0,u_1,u_2,\dotsc), y = (v_0,v_1,v_2,\dotsc) \in V_{\Gcal}$.
The set $\prod_{i = 0}^{\infty}V_i$ is equipped with the product topology.
\begin{nota}\label{nota:open-sets-of-vertices}
Let $X = V_{\Gcal}$,
 and let us define a map $f : X \to X$ by $f(x) = y$ iff $(x,y) \in E(\Gcal)$.
For each $n \benonne$, the projection from $X$ to $V_n$ is denoted by $\fai_{\infty,n}$. 
For $v \in V_n$, we denote a clopen set $U(v) := \fai_{\infty,n}^{-1}(v)$.
For a subset $K \subset V_n$, we denote a clopen set $U(K) := \bigcup_{v \in K}U(v)$.
Let $\seb n_k \sen_{k \bpi}$ be a strictly increasing sequence.
Suppose that there exist a sequence
 $\seb u_{n_k} \mid u_{n_k} \in V_{n_k}, k \bpi\sen$ such that
 $\fai_{n_{k+1},n_k}(u_{n_{k+1}})= u_{n_k}$ for all $k \bpi$.
Then, there exists a unique element $x \in V_{\Gcal}$ such that $x \in U(u_{n_k})$
 for all $k \bpi$.
This element is denoted as $x = \lim_{k \to \infty}u_{n_k}$.
\end{nota}
%
%
%
The next follows:
\begin{thm}[Theorem 3.9 and Lemma 3.5 of \cite{Shimomura4}]\label{thm:0dim=covering}
Let $\Gcal$ be a covering
 $\covrep{G}{\fai}$.
Let $X = V_{\Gcal}$ and let us define $f : X \to X$ as above.
Then, $f$ is a continuous surjective mapping and $(X,f)$ is
 a zero-dimensional system.
Conversely, every zero-dimensional system can be written in this way.
Furthermore, if all $\fai_n$ are \bidirectional, then this zero-dimensional system is a 
 homeomorphism and every compact zero-dimensional homeomorphism is written in this way.
\end{thm}
\begin{rem}
Let $\covrep{G}{\fai}$ be a graph
 covering.
Let $(X,f)$ be the inverse limit.
For each $n \bni$, the set $\Ucal_n := \seb U(v) \mid v \in V(G_n) \sen$
 is a clopen partition
 such that $U(v) \cap f(U(u)) \nekuu$ if and only if $(u,v) \in E(G_n)$.
Furthermore, $\bigcup_{n \bni} \Ucal_n$ generates the topology of $X$.
Conversely, suppose that $\Ucal_n$ $(n \bni)$ be a sequence of finite clopen partitions
 of a compact metrizable zero-dimensional space $X$, $\bigcup_{n \bni}\Ucal_n$ 
 generates the topology of X, and $f : X \to X$ be a
 continuous surjective map such that for any $U \in \Ucal_{n+1}$ there exists 
 $U' \in \Ucal_n$ such that $f(U) \subset U'$.
Then, we can define a graph covering in the usual manner.
\end{rem}
%
%
We sometimes write the inverse limit $(X,f)$ as $G_{\infty}$.
\begin{nota}
Let $G = (V,E)$ be a surjective directed graph.
A sequence of vertices $(v_0,v_1,\dotsc,v_l)$ of $G$ is said to be a {\it walk}\/ of
 {\it length}\/ $l$ if $(v_i, v_{i+1}) \in E$ for all $0 \le i < l$.
We denote $l(w) := l$.
In this case, for $0 \le a \le b \le l$, we denote a restricted walk as
 $w[a,b] := (v_a,v_{a+1},\dotsc,v_b)$.
We say that a walk $w = (v_0,v_1,\dotsc,v_l)$ is a {\it path}\/
 if the $v_i$ $(0 \le i \le l)$ are mutually distinct.
For a walk $w = (v_0,v_1,\dotsc,v_l)$,
 we define $V(w) := \seb v_i \mid 0 \le i \le l \sen$
 and $E(w) := \seb (v_i,v_{i+1}) \mid 0 \le i < l \sen$.
For a subgraph $G'$ of $G$, we also define $V(G')$ and $E(G')$ in the same manner.
For a walk $w$ and a subgraph $G'$, we also denote a clopen set
 $U(w) := \bigcup_{v \in V(w)}U(v)$
 and $U(G') := \bigcup_{v \in V(G')}U(v)$.
\end{nota}
%
%
\begin{nota}
Let $w_1 = (u_0,u_1,\dotsc, u_l)$ and $w_2 = (v_0,v_1,\dotsc, v_{l'})$ be walks
 such that $u_l = v_0$.
Then, we denote $w_1  w_2 := (u_0,u_1,\dots,u_l,v_1,v_2,\dotsc,v_{l'})$.
Evidently, we get $l(w_1  w_2) = l+l'$.
\end{nota}
%
%
\begin{defn}
A finite directed graph $G$ is {\it irreducible}\/ if for any pair of vertices
 $(u,v)$ of $G$, there exists a walk from $u$ to $v$.
\end{defn}
In this case, for any vertices $u,v \in V(G)$, there exists a walk $w$ from $u$ to $v$
 such that $E(w) = E(G)$.
\begin{rem}\label{rem:chain-transitive-irreducible}
Let $\covrep{G}{\fai}$ be a graph covering with the inverse limit $(X,f)$.
It is easy to check that $(X,f)$ is chain transitive if and only if for all $n \bni$
 $G_n$ is irreducible.
\end{rem}
\section{Construction and properties.}
In this section, we prove Theorem \ref{thm:main}, constructing $(X,f)$ from $(Y,g)$.
We recall that $(Y,g)$ is an arbitrary chain transitive zero-dimensional system.
We have to construct a zero-dimensional system $(X,f)$
 such that $Y \subset X$, and there exists a dense uniformly chaotic
 subset $K \subset X \setminus Y$.
The rest of the properties in Theorem \ref{thm:main} are also shown in this section.
\subsection{Construction of the system.}
Let $\Fcal : \covrep{F}{\phi}$ be a covering that expresses zero-dimensional system
 $(Y,g)$.
We assumed that $(Y,g)$ is chain transitive.
From Remark \ref{rem:chain-transitive-irreducible}, for all $n \bni$, all $F_n$ are
irreducible.
We shall construct a new covering $\Gcal : \covrep{G}{\fai}$ such that
 for each $n \bni$, $F_n$ is a subgraph of $G_n$ and $\fai_n|_{F_n} = \phi_n$.
For $n = 0$, $F_0$ is a singleton graph $(\seb v_0 \sen, \seb (v_0,v_0) \sen)$.
We let $G_0 = F_0$.
For each $n > 0$, we take and fix vertices $v_{n,1},v_{n,2} \in V(F_n)$ such that
$\phi_{n+1}(v_{n+1,i}) = v_{n,i}$ for all $i = 1,2$ and $n \bni$.
Let $Y_i$ $(i \bi)$ be copies of $Y$.
Let $\hY
 := \seb (y_i) \in \prod_{i \bi}Y_i \mid f(y_i) = y_{i+1} \text{ for all } i \bi \sen$.
Let  $x := (v_{0,0},v_{1,1},v_{2,1},v_{3,1},\dotsc),
 y := (v_{0,0},v_{1,2},v_{2,2},v_{3,2},\dotsc) \in Y$.
Take $\hx=(x_i), \hy=(y_i) \in \hY$ such that $x_0 = x$ and $y_0 = y$.
Let $w_{n,1} = (v_{n,1,0} = v_{n,1},v_{n,1,1},v_{n,1,2},\dotsc,v_{n,1,n})$
 be a walk 
 such that $x_i \in U(v_{n,1,i})$ for all $0 \le i \le n$.
We also define a vertex $v_{n,1,n+1} \in V(F_n)$ with $x_{n+1} \in U(v_{n,1,n+1})$.
Let $w_{n,2} = (v_{n,2,-n},v_{n,2,-(n-1)}, \dotsc,v_{n,2,-1},v_{n,2,0} = v_{n,2})$
 be a walk 
 such that $x_i \in U(v_{n,1,i})$ for all $-n \le i \le 0$.
We also define a vertex $v_{n,2,-(n+1)} \in V(F_n)$
 with $y_{-(n+1)} \in U(v_{n,1,-(n+1)})$.
For each $n \bni$, take a walk $w_n$ in $F_n$ from $v_{n,1,n}$ to $v_{n,2,-n}$
 such that $E(w) = E(F_n)$.
We shall attach a vertices $v_{n,0}$ for each $n > 0$ different from vertices of $F_n$,
 and paths $p_{1,n}$, $p_{2,n}$ that
 connect from $v_{n,0}$ to $v_{n,1,n}$ and from $v_{n,2,-n}$ to $v_{n,0}$ respectively
 (see Figure \ref{fig:construction1}).
The length of $p_{1,n},p_{2,n}$ and the covering map $\fai_n$ shall be defined later.
We also assume the edge $e_n = (v_{n,0},v_{n,0})$ for each $n > 0$.
Therefore, it follows that $V(G_n) = V(p_{1,n}) \cup V(p_{2,n}) \cup V(F_n)$ and
 $E(G_n) = E(p_{1,n}) \cup E(p_{2,n}) \cup \seb e_n \sen \cup E(F_n)$.

\begin{figure}\label{fig:construction1}
\begin{center}
\includegraphics[width = 0.8\textwidth, height=14cm]{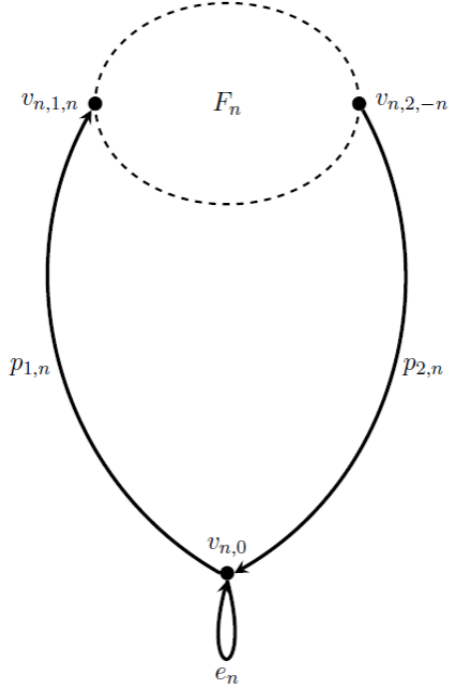}
\end{center}
\vspace{-38mm}

\caption{The way of constructing $G_n$.}
\end{figure}

We write as:
\itemb
\item $p_{1,n}
 = (v_1(n,0) = v_{n,0}, v_1(n,1), v_1(n,2), \dotsc, v_1(n,l_{1,n}) = v_{n,1,n})$,
\item $p_{2,n}
 = (v_2(n,0) = v_{n,2,-n}, v_2(n,1), v_2(n,2), \dotsc, v_2(n,l_{2,n}) = v_{n,0})$.
\itemn
We also write $p_{i,n}(j) = v_i(n,j)$ for $i = 1,2$ and $0 \le j \le l_{i,n}$.
The lengths $l_{1,n},l_{2,n}$ will be defined in the following way.
For $n = 1$, we define $l_{1,1} \ge 1$ and $l_{2,1} \ge 1$ arbitrarily.
The covering map $\fai_1$ is obviously defined.
If $G_n$ is defined, then $\fai_{n+1}$ is defined in the next way:
\itemb
\item $\fai_{n+1}(v_{n+1,0}) = v_{n,0}$,
\item $\fai_{n+1}(v) = \phi_{n+1}(v)$ for all $v \in V(F_{n+1})$.
\item $\fai_{n+1}(p_{1,n+1})
 = e_n\ (p_{1,n}\ w_n\ p_{2,n})^2
 \ p_{1,n}\ (v_{n,1,n},v_{n,1,n+1})$,
\item $\fai_{n+1}(p_{2,n+1})
 = (v_{n,2,-(n+1)},v_{n,2,-n})
 \ p_{2,n}\ (p_{1,n}\ w_n\ p_{2,n})^2\ e_n$.
\itemn
The length $l_{1,n+1},l_{2,n+1}$
 is defined such that the above definitions hold.
\begin{lem}\label{lem:well-defined}
In the above definition, all $\fai_n$'s are \pdirectionals.
If $\Fcal$ is \bidirectional, then all $\fai_n$'s are \bidirectional.
\end{lem}
\begin{proof}
Both the \pdirectionalitys and \bidirectionalitys have to be checked only at the points
 $v_{n+1,0},v_{n+1,1,n+1},v_{n+1,2,-(n+1)}$.
For the other vertices, the conclusion is obvious from the definition.
At the vertex $v_{n+1,0}$, there exists edges $e_{n+1}$, $(v_1(n+1,0),v_1(n+1,1))$
 and $(v_2(n+1,l_{2,n+1}-1),v_2(n+1,l_{2,n+1}))$.
By the definition, all of these edges are mapped to $e_n$.
Therefore, $\fai_{n+1}$ is \bidirectionals at $v_{n+1,0}$.
At the vertex $v_{n+1,1,n+1}$, the \pdirectionalitys is evident from that of $\Fcal$.
Suppose that $\Fcal$ is \bidirectional.
Let $(u_j,v_{n+1,1,n+1})$ $(j = 1,2,\dotsc,k)$ be edges in $F_{n+1}$.
Then, because of \bidirectionality, $\fai_{n+1}(u_j)$'s are identical for all $j$.
The last edge of $p_{1,n+1}$ is mapped by $\fai_{n+1}$ onto $(v_{n,1,n},v_{n,1,n+1})$.
The edge $(v_{n,1,n},v_{n,1,n+1})$
 is projected from the edge\\
 $(v_{n+1,1,n},v_{n+1,1,n+1})$.
Therefore, the \bidirectionalitys of $\fai_{n+1}$ follows at the vertex $v_{n+1,1,n+1}$.
The similar argument can be applied to the vertex $v_{n+1,2,n+1}$.
In this case, the \pdirectionalitys has to be checked carefully.
This concludes the proof.
\end{proof}

By Lemma \ref{lem:well-defined}, we can define the inverse limit $G_{\infty} = (X,f)$.
By the definition of $\Gcal$, $(Y,g)$ is embedded in $(X,f)$.
\subsection{A construction of the required chaotic subset.}
We find a fixed point $p \in X$ as 
 $p = (v_0,v_{1,0},v_{2,0},\dotsc) \in X$.
It follows that $X$ does not have isolated points.
To see this, it is enough to check that every vertex $v \in V(G_n)$ is covered
 by multiple vertices, and this fact is easily seen from the definition of $\Gcal$.
We prove Theorem \ref{thm:main} by successive lemmas.
\begin{lem}
The restriction $f|_{X \setminus Y}$ is a homeomorphism.
\end{lem}
\begin{proof}
We note that any point $x \in X \setminus Y$ is expressed by the vertices
 $V(G_n) \setminus V(F_n)$ for sufficiently large $n$.
Such vertices are \bidirectionals vertices.
Therefore, every $x \in X \setminus Y$ has the unique $f^{-i}(x)$ for all $i \ge 0$.
\end{proof}
This lemma has proved \ref{main:homeomorphism}.
%
%
%
%
%
%
\begin{lem}\label{lem:enters-around-fixed-point}
For $m > n$, it follows that 
\itemb
\item $\fai_{m,n}(p_{2,m}\ p_{1,m})
 = \quad \dotsb \quad p_{2,n}\ {e_n}^{m-n}\ {e_n}^{m-n}\ p_{1,n} \quad \dotsb\ $, and
\item $\fai_{m,n}(p_{1,m}) = {e_n}^{m-n}\ p_{1,n} \quad \dotsb$ \quad
 {\scriptsize {\rm length more than } $m-n$ } $\dotsb \quad $
 \\
 \hspace{4.3cm} $\dotsb \quad $
 $p_{2,n}\ {e_n}^{m-n-1}\ {e_n}^{m-n-1}\ p_{1,n}\quad \dotsb\ $.
\itemn
\end{lem}
\begin{proof}
The proof is obvious by the definition of $\Gcal$.
\end{proof}
Let $n > 0$ sufficiently large.
Then, 
\[
\begin{array}{ll}
 \fai_{n+1}(p_{1,n+1}) = &
           e_n\ (p_{1,n}\ w_n\ p_{2,n})^2 \ p_{1,n}\ (v_{n,1,n},v_{n,1,n+1})\\
         & e_n\ p_{1,n}\ w_n\ p_{2,n}\ \underline{p_{1,n}}\ w_n\ p_{2,n}\ p_{1,n}
 \ (v_{n,1,n},v_{n,1,n+1})
\end{array}
\]
In the left of the underlined occurrence, there exists $e_n\ p_{1,n}$.
In the right of the underlined occurrence, there exists $p_{2,n} p_{1,n}$.
Let $m > n$.
For the projection 
$\fai_{m,n}= \fai_{n+1} \circ \dotsb \circ \fai_{m-1} \circ \fai_m(p_{1,m})$,
 we select the mid $p_{1,k}$ $n \le k < m$ successively.
Therefore, we can find the mid (not necessarily exactly central)
 segment $q_m = p_{1,m}[l,l']$ that is projected
 isomorphically onto the mid occurrence of $p_{1,n}$.
The length of $q_m$ is $l_{1,n}$.
Let $m > n + l_{1,n}$.
Then, by the previous lemma, it follows that
\itemb
\item  (Property 1) : in the sequence $p_{1,m} = \ \dots\ q_m \dots$,
 both right and left translation of $q_m$ are projected by $\fai_{m,n}$ 
 into the segment ${e_n}^{l}$ with $l \ge m-n$, and 
\item  (Property 2) : in the sequence $p_{1,m} = \ \dots\ q_m \dots$,
 both right and left translations of $q_m$ are projected by $\fai_{m,n}$ 
 onto the segment $p_{1,n}$ isomorphically,
\item in the above translations, one can choose the translations such that
 the translation lengths tend to infinity as $m \to \infty$.
\itemn
Let $m(1) = m$, and we have gotten $q_{m(1)}$.
Reset $n = m(1)+1$, and we get $q_{m(2)}$.
Then, by the map $\fai_{m(2),m(1)}$, $q_{m(2)}$ covers $q_{m(1)}$ three times.
In this way, we get the sequence $q_{m(k)}$ for each $k \ge 1$.
%
%
\begin{figure}

\vspace{-4.5cm}

\begin{center}\leavevmode 
\includegraphics[width = 1.0\textwidth, height=20cm]{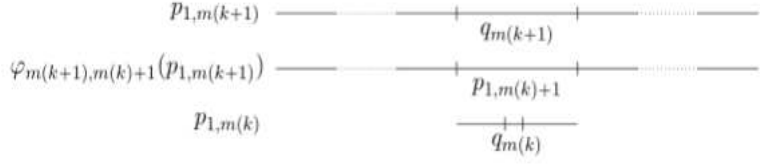}
\end{center}

\vspace{-13cm}

\caption{Constructing the walks $q_{m(k)}$ $(k \ge 1)$.}\label{fig:walks-q_k}
\end{figure}
For each $N \ge 1$, let $C_N := \bigcap_{k \ge N} U(q_{m(k)})$.
Then, because each $U(q_{m(k)})$ is closed, each $C_N$ is closed.
Because, by the map $\fai_{m(k+1),m(k)}$, $q_{m(k+1)}$ covers $q_{m(k)}$ three
 times, each $C_N$ is a Cantor set.
It is evident that $C_1 \subset C_2 \subset \dotsb$.
We define $K = \bigcup_{N > 0}C_N$.
This proves \ref{main:K}.
\begin{lem}\label{lem:not-kuu}
For each $N$, $C_N \cap U(v) \nekuu$ for each $v \in V(q_{m(k)})$
 with $k \ge N$.
\end{lem}
\begin{proof}
For each $k$, $\fai_{m(k+1),m(k)}$ maps $q_{m(k+1)}$ onto all of $q_{m(k)}$
 three times.
Therefore, by the definition of $C_N$, we get the conclusion.
\end{proof}
%
%
\begin{lem}\label{lem:C_N-is-proximal}
For each $N >0$, $C_N$ is positively and also negatively uniformly proximal.
\end{lem}
\begin{proof}
This is obvious by the Property 1.
\end{proof}
%
%
\begin{lem}\label{lem:C_N-is-recurrent}
For each $N > 0$, $C_N$ is positively and also negatively uniformly recurrent.
\end{lem}
\begin{proof}
This is obvious by the Property 2.
\end{proof}
We get \ref{main:proximal} and \ref{main:recurrent} by the above two lemmas.
The next lemma proves \ref{main:invariant}.
\begin{lem}\label{lem:invariant}
We get $K \subset X \setminus Y$, and $f(K) = K$.
\end{lem}
\begin{proof}
It follows that
\[K = \seb (v_0,v_1,v_2,\dotsc) \in X \mid 
 \exists k_0, \forall k \ge k_0,  v_{m(k)} \in V(q_{m(k)}) \sen.\]
Because $V(q_k) \subset V(G_{m(k)}) \setminus V(F_{m(k)})$,
 the first statement is obvious.
Let $x = (v_0,v_1,v_2,\dotsc) \in K$.
Then, there exists an $k_0 > 0$ such that for all $k > k_0$,
 $v_{m(k)}$ is not the edge vertex of $q_{m(k)}$.
Therefore, $f^{-1}(x),f(x) \in K$.
\end{proof}
\begin{lem}\label{lem:dense}
$K$ is dense in $X$.
\end{lem}
\begin{proof}
Take $k > 0$.
Then, $\fai_{m(k+1),m(k)+1}(q_{m(k+1)}) = p_{1,m(k)+1}$.
Therefore, it is evident that $V(\fai_{m(k+1),m(k)}(q_{k+1})) = V(G_{m(k)})$.
By Lemma \ref{lem:not-kuu}, we get the conclusion.
\end{proof}
This proves \ref{main:dense}.
As notified in Remark 2.12 of \cite{AGHSY}, we also get the following; and this
 proves \ref{main:transitive}:
\begin{lem}
Each $x \in K$ is a both positively and negatively transitive point.
\end{lem}
\begin{proof}
The proof is evident from Property 2.
\end{proof}
Finally, we note that, crashing $Y$ into a point, we get a very simple chaotic system.
It is easy to check that this last system has zero topological entropy.
Therefore, by a standard calculation, we get $h(f) = h(g)$.

In the above construction of $G_n$, we can take $w_n$ very very long compared with
 the lengths of $p_{1,n}, p_{2,n}$.
Then, except the ergodic measure on the fixed point $p$, there may not arise
 any additional ergodic measures on $(X,f)$ other than that of $(Y,g)$ (cf. \cite[Theorem 4.22]{Shimomura5}).

\end{document}